\documentclass[12pt, reqno]{amsart}
\makeatletter
\@namedef{subjclassname@1991}{$\mathrm{1991}$ Mathematics Subject Classification}
\@namedef{subjclassname@2000}{$\mathrm{2000}$ Mathematics Subject Classification}
\@namedef{subjclassname@2010}{$\mathrm{2010}$ Mathematics Subject Classification}
\@namedef{subjclassname@2020}{$\mathrm{2020}$ Mathematics Subject Classification}
\makeatother
\usepackage{amsmath,amsthm, amscd, amsfonts, amssymb, graphicx, color}
\usepackage[bookmarksnumbered, colorlinks, plainpages,linkcolor=blue,urlcolor=blue,citecolor=blue]{hyperref}
\newtheorem{thm}{Theorem}[section]
\newtheorem{cor}[thm]{Corollary}
\newtheorem{lem}[thm]{Lemma}

\newtheorem{exam}[thm]{Example}
\numberwithin{equation}{section}

\begin{document}

\title{Pseudo core invertibility in Banach *-algebras and its applications}

\author{Huanyin Chen}
\author{Marjan Sheibani$^*$}
\address{
School of Mathematics\\ Hangzhou Normal University\\ Hangzhou, China}
\email{<huanyinchenhz@163.com>}
\address{Farzanegan Campus, Semnan University, Semnan, Iran}
\email{<m.sheibani@semnan.ac.ir>}


\begin{abstract} We present new additive results for the pseudo core inverse in a Banach algebra with involution. The necessary and sufficient conditions
under which the sum of two pseudo core invertible elements in Banach *-algebra is pseudo core invertible are obtained. As an application, the pseudo core invertibility for block complex matrices is investigated. These extend the main results of pseudo core invertibility of Gao and Chen [Comm. Algebra, 46(2018), 38--50].
\end{abstract}

\thanks{Corresponding author: Marjan Sheibani}

\maketitle

\section{Introduction}

An involution of a Banach algebra $\mathcal{A}$ is an anti-automorphism whose square is the identity map $1$. A Banach algebra $\mathcal{A}$ with involution $*$ is called a Banach *-algebra, e.g., $C^*$-algebra. Let $\mathcal{A}$ be a Banach *-algebra. An element $a\in \mathcal{A}$ has p-core inverse (i.e., pseudo core inverse) if there exist $x\in \mathcal{A}$ and $k\in \Bbb{N}$ such that $$xa^{k+1}=a^k, ax^2=x, (ax)^*=ax.$$ If such $x$ exists, it is unique, and denote it by $a^{\tiny\textcircled{D}}$. An element $a\in \mathcal{A}$ has Drazin inverse
provided that there exists $x\in \mathcal{A}$ such that $$xa^{k+1}=a^k, ax^2=x, ax=xa,$$ where $k$ is the index of $a$ (denoted by $i(a)$), i.e., the smallest $k$ such that the previous equations are satisfied. Such $x$ is unique if exists, denoted by $a^{D}$, and called the Drazin inverse of $a$. As is well known, a square complex matrix $A$ has group inverse if and only if $rank(A^k)=rank(A^{k+1})$. The p-core invertibility in a Banach *-algebra is attractive. This notion was introduced by Gao and Chen in 2018 (see~\cite{{GC}}). This is a natural extension of the core inverse which is the first studied by Baksalary and Trenkler for a complex matrix in 2010 (see~\cite{BT}). A matrix $A\in C^{n\times n}$ has core inverse $A^{\tiny\textcircled{\#}}$ if and only if $AA^{\tiny\textcircled{\#}}=P_A$ and $\mathcal{R}(A^{\tiny\textcircled{\#}})\subseteq \mathcal{R}(A)$, where $P_A$ is the projection on $\mathcal{R}(A)$ (see~\cite{BT}). Rakic et al. (see~\cite{R}) generalized the core inverse of a complex to the case of an element in a ring. An element $a$ is a Banach algebra $\mathcal{A}$ has core inverse if and only if there exist $x\in \mathcal{A}$ such that $$a=axa, x\mathcal{A}=a\mathcal{A}, \mathcal{A}x=\mathcal{A}a^*.$$ If such $x$ exists, it is unique, and denote it by $a^{\tiny\textcircled{\#}}$. Recently, many authors have studied core and p-core inverses from many different views, e.g., ~\cite{CZ2,CCZ,GC2,GC3,K,M,R,XS,XS1,Z2}.
An element $a\in \mathcal{A}$ has $(1,3)$ inverse provided that there exists some $x\in \mathcal{A}$ such that $a=axa$ and $(ax)^*=ax$. We list several characterizations of p-core inverse.

\begin{thm} (see~\cite[Theorem 2.3 and Theorem 2.5]{GC}, \cite[Theorem 3.1]{XS1}) Let $\mathcal{A}$ be a Banach *-algebra, and let $a\in \mathcal{A}$. Then the following are equivalent:\end{thm}
\begin{enumerate}
\item [(1)]{\it $a\in \mathcal{A}^{\tiny\textcircled{D}}$.}
\item [(2)]{\it $a\in \mathcal{A}^D$ and $a^k$ has $(1,3)$ inverse, where $k=i(a)$.}
\item [(3)]{\it There exists $x\in \mathcal{A}$ such that $a^nxa^n=a$ and $a^nR=xR=x^*R$ for some $n\in {\Bbb N}$.}
\item [(4)]{\it $a^m\in \mathcal{A}$ has core inverse for some positive integer $m$.}
\end{enumerate}

Let $a,b\in \mathcal{A}$ have p-core inverses. In ~\cite[Theorem 4.4]{GC}, Gao and Chen proved that $a+b$ has p-core inverse when $ab=ba=0$ and $a^*b=0$. This inspires us to investigate new additive properties for p-core invertibility in a Banach *-algebra.

In Section 2, we are concerned with new additive results for p-core invertible elements in a Banach *-algebra. If $ab=ba$ and $a^*b=ba^*$, we present necessary and sufficient conditions under which $a+b\in \mathcal{A}$ is p-core invertible.

Lex $X$ be a Banach space. Denote by $\mathcal{B}(X)$ the Banach algebra of all bounded linear operators from $X$ to itself.
In Section 3,, we apply our additive results to bounded linear operators and obtain various conditions under which a block-operator matrix has p-core inverse.

Throughout the paper, all Banach *-algebras are complex with an identity. An element $p\in \mathcal{A}$ is a projection if $p^2=p=p^*$.
$\mathcal{A}^{D},\mathcal{A}^{\tiny\textcircled{D}}$ and $\mathcal{A}^{nil}$ denote the sets of all Drazin, p-core invertible and nilpotent elements in $\mathcal{A}$ respectively. Let $a\in \mathcal{A}^D$. We use $a^{\pi}$ to stand for the spectral idempotent of $a$ corresponding to $\{ 0\}$, i.e., $a^{\pi}=1-aa^D$.

\section{Key lemmas}

To prove the main results, some lemmas are needed. We begin with

\begin{lem} (~\cite[Proposition 4.2]{GC})) Let $a,b\in \mathcal{A}^{\tiny\textcircled{D}}$. If $ab=ba$ and $a^*b=ba^*$, then $a^{\tiny\textcircled{D}}b=ba^{\tiny\textcircled{D}}$. \end{lem}

\begin{lem} (~\cite[Theorem 4.3]{GC})) Let $a,b\in \mathcal{A}^{\tiny\textcircled{D}}$. If $ab=ba$ and $a^*b=ba^*$, then $ab\in \mathcal{A}^{\tiny\textcircled{D}}$ and $(ab)^{\tiny\textcircled{D}}=a^{\tiny\textcircled{D}}b^{\tiny\textcircled{D}}$.\end{lem}

\begin{lem} (~\cite[Theorem 4.4]{GC})) Let $a,b\in \mathcal{A}^{\tiny\textcircled{D}}$. If $ab=ba=0$ and $a^*b=0$, then $a+b\in \mathcal{A}^{\tiny\textcircled{D}}$.\end{lem}

\begin{lem} Let $a\in \mathcal{A}^{\tiny\textcircled{D}}$ and $b\in \mathcal{A}$. Then the following are equivalent:\end{lem}
\begin{enumerate}
\item [(1)] $(1-a^{\tiny\textcircled{D}}a)b=0$.
\vspace{-.5mm}
\item [(2)] $(1-aa^{\tiny\textcircled{D}})b=0$.
\end{enumerate}
\begin{proof} $(1)\Rightarrow (2)$ Since $(1-a^{\tiny\textcircled{D}}a)b=0$, we have $b=a^{\tiny\textcircled{D}}ab$. Hence,
$(1-aa^{\tiny\textcircled{D}})b=(1-aa^{\tiny\textcircled{D}})a^{\tiny\textcircled{D}}ab=0.$

$(2)\Rightarrow (1)$ Since $(1-aa^{\tiny\textcircled{D}})b=0$, we get $b=aa^{\tiny\textcircled{D}}b$. Therefore
$(1-a^{\tiny\textcircled{D}}a)b=(1-a^{\tiny\textcircled{D}}a))aa^{\tiny\textcircled{D}}b=(1-a^{\tiny\textcircled{D}}a))a^{m+1}a^D(a^m)^{(1,3)}=(a^m-a^{\tiny\textcircled{D}}a^{m+1})aa^D(a^m)^{(1,3)}
=0$, as desired.\end{proof}

Let $a, p^2=p\in \mathcal{A}$. Then $a$ has the Pierce decomposition relative to $p$: $pap+pap^{\pi}+p^{\pi}ap+p^{\pi}ap^{\pi}$. We denote it by a matrix form: $\left(\begin{array}{cc}
pap&app^{\pi}\\
p^{\pi}ap&p^{\pi}ap^{\pi}
\end{array}
\right)_p$. We now derive

\begin{lem} \end{lem}
\begin{enumerate}
\item [(1)] Let $x=\left(
\begin{array}{cc}
a&b\\
0&d
\end{array}
\right).$ Then $x\in M_2(\mathcal{A})^{\tiny\textcircled{D}}$ and $x^{\tiny\textcircled{D}}=\left(
\begin{array}{cc}
*&*\\
0&*
\end{array}
\right)$ if and only if $a,d\in \mathcal{A}^{\tiny\textcircled{D}}$ and $\sum\limits_{i=1}^{m}a^{i-1}a^{\pi}bd^{m-i}=0$ for some $m\geq i(a)$.
\item [(2)] Let $p$ be a projection and $x=\left(
\begin{array}{cc}
a&b\\
0&d
\end{array}
\right)_p.$ Then $x\in \mathcal{A}^{\tiny\textcircled{D}}$ and $p^{\pi}x^{\tiny\textcircled{D}}p=0$ if and only if $a\in (p\mathcal{A}p)^{\tiny\textcircled{D}}, d\in (p^{\pi}\mathcal{A}p^{\pi})^{\tiny\textcircled{D}}$ and $\sum\limits_{i=1}^{m}a^{i-1}a^{\pi}bd^{m-i}=0$ for some $m\geq i(a)$.
\end{enumerate}
\begin{proof} $(1)$ $\Longrightarrow $ Since $x\in \mathcal{A}^{\tiny\textcircled{D}}$, it follows by ~\cite[Theorem 2.5]{GC} that $x^m\in \mathcal{A}^{\tiny\textcircled{\#}}$, where $m=i(x)$. Then $m\geq i(a)$. In this case, $(x^m)^{\tiny\textcircled{\#}}=(x^{\tiny\textcircled{D}})^m$ and $x^{\tiny\textcircled{D}}=x^{m-1}(x^m)^{\tiny\textcircled{\#}}$.
By hypothesis, we can write $x^{\tiny\textcircled{D}}=\left(
\begin{array}{cc}
*&*\\
0&*
\end{array}
\right),$ and so $(x^{\tiny\textcircled{D}})^m=\left(
\begin{array}{cc}
*&*\\
0&*
\end{array}
\right).$ This implies that $(x^m)^{\tiny\textcircled{\#}}=\left(
\begin{array}{cc}
*&*\\
0&*
\end{array}
\right).$ Obviously, we have $x^m=\left(
\begin{array}{cc}
a^m&b_m\\
0&d^m
\end{array}
\right)_p,$ where $b_1=b, b_m=ab_{m-1}+bd^{m-1}.$
By induction, we get $b_m=\sum\limits_{i=1}^{m}a^{i-1}bd^{m-i}$.
In light of~\cite[Theorem 2.5]{XS}, $a^m,d^m\in \mathcal{A}^{\tiny\textcircled{\#}}$ and $(a^m)^{\pi}b_m=0$. By virtue of ~\cite[Theorem 2.5]{GC}, $a,d\in \mathcal{A}^{\tiny\textcircled{D}}$. Moreover, we have $\sum\limits_{i=1}^{m}a^{i-1}a^{\pi}bd^{m-i}=0$.

$\Longleftarrow $ Since $a\in (p\mathcal{A}p)^{\tiny\textcircled{D}}, d\in (p^{\pi}\mathcal{A}p^{\pi})^{\tiny\textcircled{D}}$, we easily check that $a^k\in (p\mathcal{A}p)^{\tiny\textcircled{\#}}, d^k\in (p^{\pi}\mathcal{A}p^{\pi})^{\tiny\textcircled{\#}}$, where $k=max\{m,i(b)\}$. Write $x^k=\left(
\begin{array}{cc}
a^k&b_k\\
0&d^k
\end{array}
\right),$ where $b_1=b, b_k=ab_{k-1}+bd^{k-1}.$ As in the preceding discussion,
$a^{\pi}b_k=0$. Then $x^k\in \mathcal{A}^{\tiny\textcircled{\#}}$ by~\cite[Theorem 2.5]{XS}. According to~\cite[Theorem 2.5]{GC}, we prove that $x\in \mathcal{A}^{\tiny\textcircled{D}}$.
Moreover, we have $(x^{\tiny\textcircled{D}})^k=(x^k)^{\tiny\textcircled{\#}}=\left(
\begin{array}{cc}
*&*\\
0&*
\end{array}
\right)_p$. Then $$x^{\tiny\textcircled{D}}=x^{k-1}(x^{\tiny\textcircled{D}})^k=\left(
\begin{array}{cc}
*&*\\
0&*
\end{array}
\right),$$ as asserted.

$(2)$ $\Longrightarrow $ Since $x\in \mathcal{A}^{\tiny\textcircled{D}}$, we have $x^m\in \mathcal{A}^{\tiny\textcircled{\#}}$, where $m=i(x)\geq i(a)$. Moreover,
$(x^m)^{\tiny\textcircled{\#}}=(x^{\tiny\textcircled{D}})^m$ and $x^{\tiny\textcircled{D}}=x^{m-1}(x^m)^{\tiny\textcircled{\#}}$.
Since $p^{\pi}x^{\tiny\textcircled{D}}p=0$, we can write $x^{\tiny\textcircled{D}}=\left(
\begin{array}{cc}
*&*\\
0&*
\end{array}
\right)_p,$ and so $(x^{\tiny\textcircled{D}})^m=\left(
\begin{array}{cc}
*&*\\
0&*
\end{array}
\right)_p.$ This implies that $p^{\pi}(x^{\tiny\textcircled{D}})^mp=0.$ Hence $p^{\pi}(x^m)^{\tiny\textcircled{\#}}p=0.$
Write $x^m=\left(
\begin{array}{cc}
a^m&b_m\\
0&d^m
\end{array}
\right)_p,$ where $b_1=b, b_m=ab_{m-1}+bd^{m-1}.$
Then $b_m=\sum\limits_{i=1}^{m}a^{i-1}bd^{m-i}$.
In light of~\cite[Theorem 2.5]{XS}, $a^m,d^m\in p\mathcal{A}^{\tiny\textcircled{\#}}p\subseteq (p\mathcal{A}p)^{\tiny\textcircled{\#}}$ and $(a^m)^{\pi}b_m=0$. In view of ~\cite[Theorem 2.5]{GC}, $a\in (p\mathcal{A}p)^{\tiny\textcircled{D}}$ and $d\in (p^{\pi}\mathcal{A}p^{\pi})^{\tiny\textcircled{D}}$.
Moreover, we check that $\sum\limits_{i=1}^{m}a^{i-1}a^{\pi}bd^{m-i}=0$.

$\Longleftarrow $ Since $a,d\in \mathcal{A}^{\tiny\textcircled{D}}$, it follows by~\cite[Theorem 2.5]{GC} that $a^k,d^k\in \mathcal{A}^{\tiny\textcircled{\#}}$, where $k=max\{m,i(b)\}$. Write $x^k=\left(
\begin{array}{cc}
a^k&b_k\\
0&d^k
\end{array}
\right)_p,$ where $b_1=b, b_k=ab_{k-1}+bd^{k-1}.$ Then $b_m=\sum\limits_{i=1}^{m}a^{i-1}bd^{m-i}$.
By hypothesis, we have $a^{\pi}b_m=0$. As in the preceding discussion,
we have $a^{\pi}b_k=0$. In light of~\cite[Theorem 2.5]{XS}, $x^k\in \mathcal{A}^{\tiny\textcircled{\#}}$. According to~\cite[Theorem 2.5]{GC}, we get $x\in \mathcal{A}^{\tiny\textcircled{D}}$. Further, $p^{\pi}(x^k)^{\tiny\textcircled{\#}}p=0$, and so $p^{\pi}(x^{\tiny\textcircled{D}})^kp=0$,
Write $(x^{\tiny\textcircled{D}})^k=\left(
\begin{array}{cc}
*&*\\
0&*
\end{array}
\right)_p$. Then $x^{\tiny\textcircled{D}}=x^{k-1}(x^{\tiny\textcircled{D}})^k=\left(
\begin{array}{cc}
*&*\\
0&*
\end{array}
\right)_p.$ This implies that $p^{\pi}x^{\tiny\textcircled{D}}p=0$, as asserted.\end{proof}

\section{Main Results}

This section is devoted to investigate the p-core inverse of the sum of two p-core invertible elements in a Banach *-algebra. We are ready to prove:

\begin{thm} Let $a,b\in \mathcal{A}^{\tiny\textcircled{D}}$. If $ab=ba$ and $a^*b=ba^*$, then the following are equivalent:\end{thm}
\begin{enumerate}
\item [(1)] $a+b\in \mathcal{A}^{\tiny\textcircled{D}}$ and $a^{\pi}(a+b)^{\tiny\textcircled{D}}aa^{\tiny\textcircled{D}}=0$.
\vspace{-.5mm}
\item [(2)] $1+a^{\tiny\textcircled{D}}b\in \mathcal{A}^{\tiny\textcircled{D}}$ and  $$\sum\limits_{i=1}^{m}[1+a^{\tiny\textcircled{D}}b]^{i-1}a^{i-1}[1+a^{\tiny\textcircled{D}}b]^{\pi}a[aa^{\tiny\textcircled{D}}-a^{\tiny\textcircled{D}}a](a+b)^{m-i}=0$$ for some $m\geq i(1+a^{\tiny\textcircled{D}}b)$.
\end{enumerate}
\begin{proof} Since $ab=ba$ and $a^*b=ba^*$, it follows by Lemma 2.1 that $a^{\tiny\textcircled{D}}b=ba^{\tiny\textcircled{D}}$.
Let $p=aa^{\tiny\textcircled{D}}$. Then $p^{\pi}bp=(1-aa^{\tiny\textcircled{D}})baa^{\tiny\textcircled{D}}=(1-aa^{\tiny\textcircled{D}})aa^{\tiny\textcircled{D}}b=0$. Moreover, we have
$pbp^{\pi}= aa^{\tiny\textcircled{D}}b(1-aa^{\tiny\textcircled{D}})=aba^{\tiny\textcircled{D}}(1-aa^{\tiny\textcircled{D}}=0$. We compute that $$\begin{array}{rll}
p^{\pi}ap&=&(1-aa^{\tiny\textcircled{D}})aaa^{\tiny\textcircled{D}}\\
&=&(1-aa^{\tiny\textcircled{D}})aa^D(a^m)(a^m)^{(1,3)}\\
&=&(1-aa^{\tiny\textcircled{D}})a^{m+1}(a^D)^m(a^m)(a^m)^{(1,3)}\\
&=&(a^{m+1}-aa^{\tiny\textcircled{D}}a^{m+1})(a^D)^m(a^m)(a^m)^{(1,3)}\\
&=&0.
\end{array}$$
So we get $$a=\left(
\begin{array}{cc}
a_1&a_2\\
0&a_4
\end{array}
\right)_p, b=\left(
\begin{array}{cc}
b_1&0\\
0&b_4
\end{array}
\right)_p.$$
Hence $$a+b=\left(
\begin{array}{cc}
a_1+b_1&a_2\\
0&a_4+b_4
\end{array}
\right)_p.$$

$(i)$ Clearly, we check that $$\begin{array}{rll}
(1-a^{\tiny\textcircled{D}}a)a^2a^{\tiny\textcircled{D}}&=&(1-a^{\tiny\textcircled{D}}a)a^2a^Da^m(a^m)^{(1,3)}\\
&=&(1-a^{\tiny\textcircled{D}}a)a^{m+2}a^D(a^m)^{(1,3)}\\
&=&(a^{m}-a^{\tiny\textcircled{D}}a^{m+1})a^2a^D(a^m)^{(1,3)}\\
&=&0.
\end{array}$$ In view of Lemma 2.4, $(1-aa^{\tiny\textcircled{D}})a^2a^{\tiny\textcircled{D}}=0$.
Hence $$a_1=aa^{\tiny\textcircled{D}}aaa^{\tiny\textcircled{D}}=a^2a^{\tiny\textcircled{D}}.$$

Obviously, $(1-aa^{\tiny\textcircled{D}})baa^{\tiny\textcircled{D}}=b(1-aa^{\tiny\textcircled{D}})aa^{\tiny\textcircled{D}}=0$.
It follows by Lemma 2.4 that $(1-a^{\tiny\textcircled{D}}a)baa^{\tiny\textcircled{D}}=0$. Hence we have $b_1=aa^{\tiny\textcircled{D}}baa^{\tiny\textcircled{D}}=baa^{\tiny\textcircled{D}}=a^{\tiny\textcircled{D}}abaa^{\tiny\textcircled{D}}=a^{\tiny\textcircled{D}}ba^2a^{\tiny\textcircled{D}}$, and then
$$a_1+b_1=(1+a^{\tiny\textcircled{D}}b)a^2a^{\tiny\textcircled{D}}\in \mathcal{A}^{\tiny\textcircled{D}}.$$
This implies that
$$(a_1+b_1)^{i-1}=(1+a^{\tiny\textcircled{D}}b)^{i-1}(a^2a^{\tiny\textcircled{D}})^{i-1}=(1+a^{\tiny\textcircled{D}}b)^{i-1}a^ia^{\tiny\textcircled{D}}.$$
Furthermore, $$(a_1+b_1)^D=(1+a^{\tiny\textcircled{D}}b)^Da^{\tiny\textcircled{D}}.$$ Thus $$(a_1+b_1)^{\pi}=1-(1+a^{\tiny\textcircled{D}}b)(1+a^{\tiny\textcircled{D}}b)^Daa^{\tiny\textcircled{D}}.$$

$(ii)$ Clearly, we have $(1-aa^{\tiny\textcircled{D}})aaa^{\tiny\textcircled{D}}=a^2a^{\tiny\textcircled{D}}-aa^{\tiny\textcircled{D}}aaa^{\tiny\textcircled{D}}=0.$ Then
$$a_4=(1-aa^{\tiny\textcircled{D}})a(1-aa^{\tiny\textcircled{D}})=(1-aa^{\tiny\textcircled{D}})a.$$ Hence $a_4^{m+1}=(1-aa^{\tiny\textcircled{D}})a^{m+1}=0$, and so
$a_4\in \mathcal{A}^{nil}$. Moreover, $$b_4=(1-aa^{\tiny\textcircled{D}})b(1-aa^{\tiny\textcircled{D}})=(1-aa^{\tiny\textcircled{D}})b.$$
Since $bp^{\pi}=p^{\pi}b, b^*p^{\pi}=(p^{\pi}b)^*=(bp^{\pi})^*=p^{\pi}b^*$. In light of Lemma 2.2, $b_4=p^{\pi}b\in \mathcal{A}^{\tiny\textcircled{D}}$ and $b_4^{\tiny\textcircled{D}}=p^{\pi}b^{\tiny\textcircled{D}}$.

$$a_4+b_4=(1-aa^{\tiny\textcircled{D}})(a+b)$$
$$(a_4+b_4)^{m-i}=(1-aa^{\tiny\textcircled{D}})(a+b)^{m-i}.$$

$(1)\Rightarrow (2)$ We have $$(a+b)^{\tiny\textcircled{D}}=\left(
\begin{array}{cc}
\alpha &\beta\\
0&\gamma
\end{array}
\right)_p.$$
Then $[p(a+b)p]^{\tiny\textcircled{D}}=\alpha $. That is, $(a+b)aa^{\tiny\textcircled{D}}\in \mathcal{A}^{\tiny\textcircled{D}}$.

We observe that $$\begin{array}{rll}
1+a^{\tiny\textcircled{D}}b&=&[1-aa^{\tiny\textcircled{D}}]+[aa^{\tiny\textcircled{D}}+a^{\tiny\textcircled{D}}b]\\
&=&[1-aa^{\tiny\textcircled{D}}]+[aa^{\tiny\textcircled{D}}+ba^{\tiny\textcircled{D}}]\\
&=&[1-aa^{\tiny\textcircled{D}}]+ [a+b]a^{\tiny\textcircled{D}}
\end{array}$$

We easily check that $[(a+b)aa^{\tiny\textcircled{D}}]a^{\tiny\textcircled{D}}=a^{\tiny\textcircled{D}}[(a+b)aa^{\tiny\textcircled{D}}]$. In view of Lemma 2.2,
$(a+b)a^{\tiny\textcircled{D}}=[(a+b)aa^{\tiny\textcircled{D}}]a^{\tiny\textcircled{D}}\in \mathcal{A}^{D}$. Then

$$[(a+b)aa^{\tiny\textcircled{D}}]^m=[(a+b)aa^{\tiny\textcircled{D}}]^my[(a+b)aa^{\tiny\textcircled{D}}]^m,$$
$$\big([(a+b)aa^{\tiny\textcircled{D}}]^my\big)^*=[(a+b)aa^{\tiny\textcircled{D}}]^my.$$
By induction, we have $$[(a+b)a^{\tiny\textcircled{D}}]^m[a^2a^{\tiny\textcircled{D}}]^m=[(a+b)aa^{\tiny\textcircled{D}}]^m.$$
We verify that
$$\begin{array}{rl}
&[(a+b)a^{\tiny\textcircled{D}}]^m[(a^2a^{\tiny\textcircled{D}})^my][(a+b)a^{\tiny\textcircled{D}}]^m[a^2a^{\tiny\textcircled{D}}]^m\\
=&[(a+b)aa^{\tiny\textcircled{D}}]^my[(a+b)aa^{\tiny\textcircled{D}}]^m\\
=&[(a+b)aa^{\tiny\textcircled{D}}]^m\\
=&[(a+b)a^{\tiny\textcircled{D}}]^m[a^2a^{\tiny\textcircled{D}}]^m.
\end{array}$$
Clearly, $[a^2a^{\tiny\textcircled{D}}]^m(a^{\tiny\textcircled{D}})^m=aa^{\tiny\textcircled{D}}$. Then
$$\begin{array}{rl}
&[(a+b)a^{\tiny\textcircled{D}}]^m[(a^2a^{\tiny\textcircled{D}})^my][(a+b)a^{\tiny\textcircled{D}}]^m\\
=&[(a+b)a^{\tiny\textcircled{D}}]^m,\\
&[(((a+b)a^{\tiny\textcircled{D}})^m(a^2a^{\tiny\textcircled{D}})^my)]^*\\
=&[((a+b)aa^{\tiny\textcircled{D}})^my]^*\\
=&((a+b)aa^{\tiny\textcircled{D}})^my\\
=&[(a+b)a^{\tiny\textcircled{D}}]^m(a^2a^{\tiny\textcircled{D}})^my.
\end{array}$$
Therefore $[(a+b)a^{\tiny\textcircled{D}}]^m$ has $(1,3)$-inverse $(a^2a^{\tiny\textcircled{D}})^my$.
By virtue of~\cite[Theorem 2.3]{GC}, $(a+b)a^{\tiny\textcircled{D}}\in \mathcal{A}^{\tiny\textcircled{D}}$. Obviously, we have $$[1-aa^{\tiny\textcircled{D}}][a+b]a^{\tiny\textcircled{D}}=[1-aa^{\tiny\textcircled{D}}]^*[a+b]a^{\tiny\textcircled{D}}=0.$$
According to Lemma 2.3, $1+a^{\tiny\textcircled{D}}b\in \mathcal{A}^{\tiny\textcircled{D}}.$
Moreover, we have $$\sum\limits_{i=1}^{m}[a_1+b_1]^{i-1}(a_1+b_1)^{\pi}a_2(a_4+b_4)^{m-i}=0$$ for some $m\geq i(a_1+b_1)$.
Therefore $$\sum\limits_{i=1}^{m}[1+a^{\tiny\textcircled{D}}b]^{i-1}a^{i-1}[1+a^{\tiny\textcircled{D}}b]^{\pi}a[aa^{\tiny\textcircled{D}}-a^{\tiny\textcircled{D}}a](a+b)^{m-i}=0$$ for some $m\geq i(1+a^{\tiny\textcircled{D}}b)$.

$(2)\Rightarrow (1)$ Let $x=(1+a^{\tiny\textcircled{D}}b)^{\tiny\textcircled{D}}$. Since $(1+a^{\tiny\textcircled{D}}b)aa^{\tiny\textcircled{D}}=aa^{\tiny\textcircled{D}}(1+a^{\tiny\textcircled{D}}b)$ and $(aa^{\tiny\textcircled{D}})^*=aa^{\tiny\textcircled{D}}$, we have $$aa^{\tiny\textcircled{D}}(1+a^{\tiny\textcircled{D}}b)^*=(1+a^{\tiny\textcircled{D}}b)^*aa^{\tiny\textcircled{D}}.$$
In light of Lemma 2.1, we get $aa^{\tiny\textcircled{D}}x=xaa^{\tiny\textcircled{D}}.$

Step 1. It is easy to verify that $$\begin{array}{rll}
(a^2a^{\tiny\textcircled{D}})a^{\tiny\textcircled{D}}&=&aa^{\tiny\textcircled{D}}=a^{\tiny\textcircled{D}}(a^2a^{\tiny\textcircled{D}}),\\
a^{\tiny\textcircled{D}}(a^2a^{\tiny\textcircled{D}})a^{\tiny\textcircled{D}}&=&a^{\tiny\textcircled{D}}(aa^{\tiny\textcircled{D}})=a^{\tiny\textcircled{D}},\\
(a^2a^{\tiny\textcircled{D}})a^{\tiny\textcircled{D}}(a^2a^{\tiny\textcircled{D}})&=&(aa^{\tiny\textcircled{D}})(a^2a^{\tiny\textcircled{D}})
=a^2a^{\tiny\textcircled{D}}
\end{array}$$
Thus $a^2a^{\tiny\textcircled{D}}\in \mathcal{A}^{D}$. In view of Theorem 1.1, $1+a^{\tiny\textcircled{D}}b\in \mathcal{A}^{D}$.
Since $(1+a^{\tiny\textcircled{D}}b)a^2a^{\tiny\textcircled{D}}=(a+b)aa^{\tiny\textcircled{D}}=aa^{\tiny\textcircled{D}}(a+b)=a^2a^{\tiny\textcircled{D}}(1+a^{\tiny\textcircled{D}}b)$, it follows by ~\cite[Lemma 2]{ZH} that
$(1+a^{\tiny\textcircled{D}}b)a^2a^{\tiny\textcircled{D}}\in \mathcal{A}^{D}$ and $$
\begin{array}{rl}
&[(a+b)aa^{\tiny\textcircled{D}}]^{\pi}\\
=&[(1+a^{\tiny\textcircled{D}}b)a^2a^{\tiny\textcircled{D}}]^{\pi}\\
=&1-(1+a^{\tiny\textcircled{D}}b)a^2a^{\tiny\textcircled{D}}(1+a^{\tiny\textcircled{D}}b)^{D}a^{\tiny\textcircled{D}}\\
=&1-(1+a^{\tiny\textcircled{D}}b)(1+a^{\tiny\textcircled{D}}b)^{D}aa^{\tiny\textcircled{D}}.
\end{array}$$

Step 2. We check that
$$\begin{array}{rl}
&[(1+a^{\tiny\textcircled{D}}b)a^2a^{\tiny\textcircled{D}}]^k[(a^{\tiny\textcircled{D}})^kx]\\
=&[(1+a^{\tiny\textcircled{D}}b)]^k[a^2a^{\tiny\textcircled{D}}]^k[(a^{\tiny\textcircled{D}})^kx]\\
=&[(1+a^{\tiny\textcircled{D}}b)]^kx][aa^{\tiny\textcircled{D}}]\\
\end{array}$$ Hence,
$$\begin{array}{rl}
&\big([(1+a^{\tiny\textcircled{D}}b)a^2a^{\tiny\textcircled{D}}]^k[(a^{\tiny\textcircled{D}})^kx]\big)^*\\
=&[(1+a^{\tiny\textcircled{D}}b)]^kx]^*[aa^{\tiny\textcircled{D}}]^*\\
=&[(1+a^{\tiny\textcircled{D}}b)]^kx][aa^{\tiny\textcircled{D}}]\\
=&[(1+a^{\tiny\textcircled{D}}b)a^2a^{\tiny\textcircled{D}}]^k[(a^{\tiny\textcircled{D}})^kx].
\end{array}$$ Moreover, we have
$$\begin{array}{rl}
&[(1+a^{\tiny\textcircled{D}}b)a^2a^{\tiny\textcircled{D}}]^k[(a^{\tiny\textcircled{D}})^kx][(1+a^{\tiny\textcircled{D}}b)a^2a^{\tiny\textcircled{D}}]^k\\
=&[(1+a^{\tiny\textcircled{D}}b)]^kx][aa^{\tiny\textcircled{D}}][(1+a^{\tiny\textcircled{D}}b)a^2a^{\tiny\textcircled{D}}]^k\\
=&[(1+a^{\tiny\textcircled{D}}b)]^kx[(1+a^{\tiny\textcircled{D}}b]^k[aa^{\tiny\textcircled{D}}][a^2a^{\tiny\textcircled{D}}]^k\\
=&[(1+a^{\tiny\textcircled{D}}b)]^k[a^2a^{\tiny\textcircled{D}}]^k\\
=&[(1+a^{\tiny\textcircled{D}}b)a^2a^{\tiny\textcircled{D}}]^k.
\end{array}$$
Accordingly, $(a+b)aa^{\tiny\textcircled{D}}=(1+a^{\tiny\textcircled{D}}b)a^2a^{\tiny\textcircled{D}}\in \mathcal{A}^{\tiny\textcircled{D}}.$

Step 3. Case 1. $b_4\in \mathcal{A}^{nil}$. It is easy to verify that $a_4b_4=(1-aa^{\tiny\textcircled{D}})ab=(1-aa^{\tiny\textcircled{D}})ba=b_4a_4$. Therefore $a_4+b_4\in \mathcal{A}^{nil}\subseteq \mathcal{A}^{\tiny\textcircled{D}}.$

Case 2. $b_4\not \in \mathcal{A}^{nil}$.

Let $q=b_4b_4^{\tiny\textcircled{D}}$. Then $p^{\pi}bp=(1-aa^{\tiny\textcircled{D}})baa^{\tiny\textcircled{D}}=(1-aa^{\tiny\textcircled{D}})aba^{\tiny\textcircled{D}}=0$. Similarly,
$pbp^{\pi}=0$. Moreover, $$\begin{array}{rll}
p^{\pi}ap&=&(1-aa^{\tiny\textcircled{D}})aaa^{\tiny\textcircled{D}}\\
&=&(1-aa^{\tiny\textcircled{D}})a(aa^{\tiny\textcircled{D}})^m\\
&=&(1-aa^{\tiny\textcircled{D}})a^{m+1}a^{\tiny\textcircled{D}}\\
&=&0.
\end{array}$$
So we get $$a_4=\left(
\begin{array}{cc}
c_1&0\\
0&c_4
\end{array}
\right)_q, b_4=\left(
\begin{array}{cc}
d_1&d_2\\
0&d_4
\end{array}
\right)_q.$$
Hence $a_4+b_4=\left(
\begin{array}{cc}
c_1+d_1&d_2\\
0&c_4+d_4
\end{array}
\right)_q.$ Here $a_1+b_1=(a+b)aa^{\tiny\textcircled{D}}, a_4+b_4=p^{\pi}(a+b)p^{\pi}=p^{\pi}(a+b).$

Step 1. $c_1\in \mathcal{A}^{nil}, b_1\in \mathcal{A}^{\tiny\textcircled{D}}$. By the preceding discussion, we have $c_1+d_1\in \mathcal{A}^{\tiny\textcircled{D}}.$

Step 2. Clearly, $(1-aa^{\tiny\textcircled{D}})a(1-aa^{\tiny\textcircled{D}})=(1-aa^{\tiny\textcircled{D}})a$. Hence,
$a_4\in \mathcal{A}^{nil}$. As $(1-aa^{\tiny\textcircled{D}})b(1-aa^{\tiny\textcircled{D}})=(1-aa^{\tiny\textcircled{D}})b$, we see that $b_4\in \mathcal{A}^{nil}$.
Moreover, $a_4b_4=(1-aa^{\tiny\textcircled{D}})ab=(1-aa^{\tiny\textcircled{D}})ba=b_4a_4$. Therefore $a_4+b_4\in \mathcal{A}^{nil}.$

Step 3. $$\begin{array}{ll}
&\sum\limits_{i=1}^{m}(c_1+d_1)^{i-1}(c_1+d_1)^{\pi}d_2(c_4+d_4)^{m-i}\\
=&(1-bb^{\tiny\textcircled{D}})(1-aa^{\tiny\textcircled{D}})\sum\limits_{i=1}^{m}[1+a^{\tiny\textcircled{D}}b]^{i-1}a^{i-1}[1+a^{\tiny\textcircled{D}}b]^{\pi}\\
&a[aa^{\tiny\textcircled{D}}-a^{\tiny\textcircled{D}}a](a+b)^{m-i}\\
=&0.
\end{array}$$
Therefore $a_4+b_4\in \mathcal{A}^{\tiny\textcircled{D}}.$ Moreover, we have $$\begin{array}{rl}
&\sum\limits_{i=1}^{m}(a_1+b_1)^{i-1}(a_1+b_1)^{\pi}a_2(a_4+b_4)^{m-i}\\
=&\sum\limits_{i=1}^{m}[(1+a^{\tiny\textcircled{D}}b)^{i-1}a^ia^{\tiny\textcircled{D}}][1-(1+a^{\tiny\textcircled{D}}b)(1+a^{\tiny\textcircled{D}}b)^{\#}aa^{\tiny\textcircled{D}}]
aa^{\tiny\textcircled{D}}a\\
&(1-aa^{\tiny\textcircled{D}})(a+b)^{m-i}\\
=&-\sum\limits_{i=1}^{m}[1+a^{\tiny\textcircled{D}}b]^{i-1}a^{i-1}[1+a^{\tiny\textcircled{D}}b]^{\pi}a[aa^{\tiny\textcircled{D}}-a^{\tiny\textcircled{D}}a](a+b)^{m-i}\\
=&0.
\end{array}$$
Therefore $a+b\in \mathcal{A}^{\tiny\textcircled{D}}$. Moreover, we have $p^{\pi}(a+b)^{\tiny\textcircled{D}}p=0$. In view of Lemma 2.4,  $a^{\pi}(a+b)^{\tiny\textcircled{D}}aa^{\tiny\textcircled{D}}=0$. This completes the proof.
\end{proof}

Recall that $a\in \mathcal{A}$ be *-DMP, if there exists some $n\in {\Bbb N}$ such that $a^n$ has More-Penrose and group inverses and $(a^n)^{\dag}=(a^n)^{\#}$ (see~\cite{CZ}). We now derive

\begin{cor} Let $a\in \mathcal{A}$ be *-DMP, $b\in \mathcal{A}^{\tiny\textcircled{D}}$. If $ab=ba$ and $a^*b=ba^*$, then the following are equivalent:\end{cor}
\begin{enumerate}
\item [(1)] $a+b\in \mathcal{A}^{\tiny\textcircled{D}}$.
\vspace{-.5mm}
\item [(2)] $1+a^{\tiny\textcircled{D}}b\in \mathcal{A}^{\tiny\textcircled{D}}$.
\end{enumerate}
\begin{proof} Since $a\in \mathcal{A}$ is *-DMP, it follows by~\cite[Theorem 2.10]{GCK} that $aa^{\tiny\textcircled{D}}=a^{\tiny\textcircled{D}}a$. The result follows by Theorem 3.1.\end{proof}

The preceding conditions $ab=ba$ and $a^*b=ba^*$ are necessary as the following shows.

\begin{exam} Let $\mathcal{A}={\Bbb C}^{2\times 2}$ be the Banach *-algebra of $2\times 2$ complex matrices, with conjugate transpose as the involution. Choose $$a=\left(
  \begin{array}{cc}
    i&0\\
    0&0
  \end{array}
\right), b=\left(
  \begin{array}{cc}
    0&0\\
    1&0
  \end{array}
\right)\in \mathcal{A}.$$ Then $a$ is *-DMP, $a^{\tiny\textcircled{D}}=\left(
  \begin{array}{cc}
    -i&0\\
    0&0
  \end{array}
\right), b^{\tiny\textcircled{D}}=0$ and $1+a^{\tiny\textcircled{D}}b=1\in \mathcal{A}^{\tiny\textcircled{D}}.$
Then $a+b=\left(
  \begin{array}{cc}
    i&0\\
    1&0
  \end{array}
\right)\not\in \mathcal{A}^{\tiny\textcircled{D}}.$ In this case, $ab\neq ba$.\end{exam}

\section{applications}

Lex $X$ and $Y$ be Banach spaces. Denote by $\mathcal{B}(X,Y)$ the set of all bounded linear operators from $X$ to $Y$. Let $\mathcal{B}(X)$ denote the set of all bounded linear operators from $X$ to itself. The aim of this section is to explore the the p-core invertibility of a block-operator matrix $M=\left(
\begin{array}{cc}
A&B\\
C&D
\end{array}
\right),$ where $A\in \mathcal{B}(X)^{\tiny\textcircled{D}},B\in \mathcal{B}(X,Y),C\in \mathcal{B}(Y,X),D\in \mathcal{B}(Y)^{\tiny\textcircled{D}}$ and $BC\in \mathcal{B}(X)^{\tiny\textcircled{D}}, CB\in \mathcal{B}(Y)^{\tiny\textcircled{D}}$. Here, $M$ is a bounded linear operator on $X\oplus Y$. For the detailed formula of $M^{\tiny\textcircled{D}}$, we leave to the readers as they can be derived by the straightforward computation according to our proof.

\begin{thm} If $AB=BD, DC=CA, A^*B=BD^*, D^*C=CA^*$ and $A^{\tiny\textcircled{D}}BD^{\tiny\textcircled{D}}C$ is nilpotent, then $M$ has p-core inverse.\end{thm}
\begin{proof} Write $M=P+Q$, where $$P=\left(
  \begin{array}{cc}
    A & 0 \\
    0 & D
  \end{array}
\right), Q=\left(
  \begin{array}{cc}
   0 & B \\
   C & 0
  \end{array}
\right).$$ Since $A$ and $D$ have p-core inverses, so has $P$, and that $$P^{\tiny\textcircled{D}}=\left(
  \begin{array}{cc}
    A^{\tiny\textcircled{D}} & 0 \\
    0 & D^{\tiny\textcircled{D}}
  \end{array}
\right).$$ By hypothesis, $Q^2=\left(
  \begin{array}{cc}
   BC & 0 \\
   0 & CB
  \end{array}
\right)$ has p-core inverse. In light of~\cite[Theorem 2.6]{GC},
$Q$ has p-core inverse. We easily check that
$$PQ=\left(
  \begin{array}{cc}
   0 & 0 \\
    DC & 0
  \end{array}
\right)=\left(
  \begin{array}{cc}
  0 & 0 \\
    CA & 0
  \end{array}
\right)=QP.$$ Likewise, we verify that $P^*Q=QP^*$. Moreover, we check that
$$\begin{array}{rll}
I_{X\oplus Y}+P^{\tiny\textcircled{D}}Q&=&\left(
\begin{array}{cc}
I_X&A^{\tiny\textcircled{D}}B\\
D^{\tiny\textcircled{D}}C&I_Y
\end{array}
\right).
\end{array}$$ Since $A^{\tiny\textcircled{D}}BD^{\tiny\textcircled{D}}C$ is nilpotent, we prove that
$I_{X\oplus Y}+P^{\tiny\textcircled{D}}Q$ is invertible, and so it has p-core inverse. Additionally, $[I_{X\oplus Y}+P^{\tiny\textcircled{D}}Q]^{\pi}=0$.
According to Theorem 3.1, $M$ has p-core inverse, as asserted.
\end{proof}

Let $T\in \mathcal{B}(X,Y)$. The conjugate operator of $T$ is an operator $T^*\in \mathcal{B}(Y^*,X^*)$, where $Z^*$ denotes the dual space of the Banach space $Z$.
It is easy to see that if $T\in \mathcal{B}(X)^{\tiny\textcircled{D}}$, then $T^*\in \mathcal{B}(X^*)^{\tiny\textcircled{D}}$.

\begin{cor} If $AB=BD, DC=CA, D^*C=CA^*, A^*B=BD^*$ and $BD^{\tiny\textcircled{D}}CA^{\tiny\textcircled{D}}$ is nilpotent, then $M$ has p-core inverse.\end{cor}
\begin{proof} Obviously, $M^*=\left(
  \begin{array}{cc}
   A^* & C^* \\
   B^* & D^*
  \end{array}
\right)$. By hypothesis, we have $A^*C^*=C^*D^*,D^*B^*$ $=B^*A^*,AC^*=C^*D,DB^*=B^*A$ and $(A^*)^{\tiny\textcircled{D}}C^*(D^*)^{\tiny\textcircled{D}}B^*$ is nilpotent.
Applying Theorem 4.1 to the operator $M^*$, we prove that $M^*$ has p-core inverse.
Therefore $M$ has p-core inverse, as desired.\end{proof}

We are now ready to prove:

\begin{thm} If $AB=BD, DC=CA, B^*A=DB^*$ and $B(CB)^{\tiny\textcircled{D}}DC(BC)^{\tiny\textcircled{D}}A$ is nilpotent, then $M$ has p-core inverse.\end{thm}
\begin{proof} Write $M=P+Q$, where $$P=\left(
  \begin{array}{cc}
    A & 0 \\
    0 & D
  \end{array}
\right), Q=\left(
  \begin{array}{cc}
   0 & B \\
   C & 0
  \end{array}
\right).$$ As is the proof of Theorem 4.1, $P$ and $Q$ have p-core inverses. Moreover, we check that
$$\begin{array}{rll}
Q^*P&=&\left(
  \begin{array}{cc}
   0 & C^* \\
  B^* & 0
  \end{array}
\right)\left(
  \begin{array}{cc}
    A & 0 \\
    0 & D
  \end{array}
\right)\\
&=&\left(
  \begin{array}{cc}
    0 & C^*D \\
    B^*A & 0
  \end{array}
\right)\\
&=&\left(
  \begin{array}{cc}
    0 & AC^* \\
   DB^* & 0
  \end{array}
\right)\\
&=&\left(
  \begin{array}{cc}
    A & 0 \\
    0 & D
  \end{array}
\right)\left(
  \begin{array}{cc}
   0 & C^* \\
  B^* & 0
  \end{array}
\right)\\
&=&PQ^*.
\end{array}$$ Similarly, $QP=PQ$.
Since $(Q^2)^{\tiny\textcircled{D}}=\left(
  \begin{array}{cc}
    (BC)^{\tiny\textcircled{D}} & 0 \\
    0 & (CB)^{\tiny\textcircled{D}}
  \end{array}
\right)$, it follows by ~\cite[Theorem 2.6]{GC} that
$$\begin{array}{rll}
Q^{\tiny\textcircled{D}}&=&Q(Q^2)^{\tiny\textcircled{D}}\\
&=&\left(
  \begin{array}{cc}
   0 & B \\
   C & 0
  \end{array}
\right)\left(
  \begin{array}{cc}
    (BC)^{\tiny\textcircled{D}} & 0 \\
    0 & (CB)^{\tiny\textcircled{D}}
  \end{array}
\right)\\
&=&\left(
  \begin{array}{cc}
   0 & B(CB)^{\tiny\textcircled{D}}\\
   C (BC)^{\tiny\textcircled{D}} & 0
  \end{array}
\right).
\end{array}$$
Further, we verify that
$$\begin{array}{rll}
I_{X\oplus Y}+Q^{\tiny\textcircled{D}}P&=&I_{X\oplus Y}+\left(
  \begin{array}{cc}
   0 & B(CB)^{\tiny\textcircled{D}}\\
   C (BC)^{\tiny\textcircled{D}} & 0
  \end{array}
\right)\left(
  \begin{array}{cc}
    A & 0 \\
    0 & D
  \end{array}
\right)\\
&=&\left(
\begin{array}{cc}
I_X&B(CB)^{\tiny\textcircled{D}}D\\
 C (BC)^{\tiny\textcircled{D}}A&I_Y
\end{array}
\right).
\end{array}$$ Since $B(CB)^{\tiny\textcircled{D}}DC (BC)^{\tiny\textcircled{D}}A$ is nilpotent, we prove that
$I_{X\oplus Y}+Q^{\tiny\textcircled{D}}P$ is invertible; hence, it has p-core inverse.  Additionally, $[I_{X\oplus Y}+Q^{\tiny\textcircled{D}}P]^{\pi}=0$.
In light of Theorem 3.1, $M$ has p-core inverse.\end{proof}

\begin{cor} If $AB=BD, DC=CA, C^*A=DC^*$ and $A(CB)^{\tiny\textcircled{D}}BD(BC)^{\tiny\textcircled{D}}$ $C$ is nilpotent is nilpotent, then $M$ has p-core inverse.\end{cor}
\begin{proof} Similarly to Corollary 4.2, it is enough to apply Theorem 4.3 to the operator $M^*$.\end{proof}

\begin{thm} If $BC=0, CB=0, CA=DC, AC^*=C^*D$ and $$\sum\limits_{i=1}^{i(A)}A^{i-1}A^{\pi}BD^{m-i}=0,$$ then $M$ has p-core inverse.\end{thm}
\begin{proof} Write $M=P+Q$, where $$P=\left(
  \begin{array}{cc}
   0 & 0 \\
    C & 0
  \end{array}
\right), Q=\left(
  \begin{array}{cc}
   A & B \\
   0 & D
  \end{array}
\right).$$ Clearly, $P$ has p-core inverse. According to Lemma 2.5, $Q$ has p-core inverse.
We check that
$$\begin{array}{rll}
PQ&=&\left(
  \begin{array}{cc}
   0 & 0 \\
    CA & CB
  \end{array}
\right)\\
&=&\left(
  \begin{array}{cc}
    BC & 0 \\
    DC & 0
  \end{array}
\right)\\
&=&QP;\\
P^*Q&=&\left(
  \begin{array}{cc}
    0 & C^*D \\
    0 & 0
  \end{array}
\right)\\
&=&\left(
  \begin{array}{cc}
    0 & AC^* \\
   0 & 0
  \end{array}
\right)\\
&=&QP^*.
\end{array}$$ Clearly $P^{\tiny\textcircled{D}}=0$, and so $I_{X\oplus Y}+P^{\tiny\textcircled{D}}Q=I_{X\oplus Y}$ has p-core inverse.
Therefore $M$ has p-core inverse by Theorem 3.1.\end{proof}

\begin{cor} If $BC=0, CB=0, AB=BD, A^*B=BD^*$ and $$\sum\limits_{i=1}^{i(A)}CA^{i-1}A^{\pi}=0,$$ then $M$ has p-core inverse.\end{cor}
\begin{proof} As in the discussion in Corollary 4.2, we may apply Theorem 4.5 to the operator $M^*$.\end{proof}

\end{document}